\documentclass[leqno,12pt]{amsart} 
\usepackage{amssymb}
\usepackage{amsmath}
\usepackage{amsthm}
\usepackage{amscd}
\usepackage{graphicx}
\usepackage[all]{xy}
\usepackage[dvips]{color}
\theoremstyle{plain} 
\newtheorem{theorem}{\indent\sc Theorem}[section]
\newtheorem{lemma}[theorem]{\indent\sc Lemma}

\newtheorem{proposition}[theorem]{\indent\sc Proposition}

\theoremstyle{definition} 

\newtheorem{remark}[theorem]{\indent\sc Remark}
\newtheorem{example}[theorem]{\indent\sc Example}

\newtheorem{question}[theorem]{\indent\sc Question}
\begin{document}

\title[Arnol'd's type theorem on a neighborhood of a cycle of rational curves]
{Arnol'd's type theorem on a neighborhood of a cycle of rational curves} 

\author[T. Koike]{Takayuki Koike} 

%
\address{
Department of Mathematics, Graduate School of Science, Osaka City University \endgraf
3-3-138, Sugimoto, Sumiyoshi-ku Osaka 558-8585 \endgraf
Japan
}
\email{tkoike@sci.osaka-cu.ac.jp}

\maketitle

\begin{abstract}
Arnol'd showed the uniqueness of the complex analytic structure of a small neighborhood of a non-singular elliptic curve embedded in a non-singular surface whose normal bundle satisfies Diophantine condition in the Picard variety. 
We show an analogue of this theorem for a neighborhood of a cycle of rational curves. 
\end{abstract}

\section{Introduction}

Let $C$ be a cycle of rational curves (i.e. a reduced singular complex curve with only nodes such that the dual graph is a cycle graph and each complement of the normalization is biholomorphic to the projective line $\mathbb{P}^1$) holomorphically embedded in a non-singular complex surface $S$. 
Assume that the normal line bundle $N_{C/S}:=[C]|_C$ is topologically trivial, where $[C]$ is the holomorphic line bundle on $S$ defined by the divisor $C$. 
Denote by $t(N_{C/S})\in\mathbb{C}^*$ the complex number which corresponds to $N_{C/S}$ via the natural identification of Picard variety ${\rm Pic}^0(C)$ of $C$ with $H^1(C, \mathbb{C}^*) = \mathbb{C}^*$, where $\mathbb{C}^*=\mathbb{C}\setminus\{0\}$ (\cite[Lemma 1]{U91}, see also \S \ref{section:Pic}). 
We show the following theorem on the uniqueness of the complex analytic structure of a small neighborhood of $C$ under a Diophantine type condition for the normal bundle. 

\begin{theorem}\label{thm:arnold_for_nodal_rational}
Let $C$ be a cycle of rational curves, 
and $i\colon C\to S$ and $i'\colon C\to S'$ be holomorphic embeddings into non-singular complex surfaces $S$ and $S'$ respectively. 
Assume that $t(N_{i(C)/S})=t(N_{i'(C)/S'})=e^{2\pi\sqrt{-1}\theta}$ holds for a Diophantine irrational number $\theta\in\mathbb{R}$ (i.e. there exist positive constants $\alpha$ and $A$ such that $|n\cdot \theta-m|\geq A\cdot n^{-\alpha}$ holds for any integer $m$ and any positive integer $n$).  
Then there exists a biholomorphism $f\colon V\to V'$ between a neighborhood $V$ of $i(C)$ in $S$ and $V'$ of $i'(C)$ in $S'$ with $f|_{i(C)}=i'\circ i^{-1}$. 
\end{theorem}

Theorem \ref{thm:arnold_for_nodal_rational} can be regarded as an analogue of Arnol'd's theorem \cite{A}, which states that the conclusion of the theorem holds for a non-singular elliptic $C$ embedded in a non-singular surface $S$ under the assumption that $N_{C/S}$ satisfies the Diophantine type condition in the Picard variety. 

Note that, in our notation, $C$ is a cycle of rational curves with only one irreducible component when $C$ is a rational curve with a node. 
Neighborhoods of a rational curve with a node embedded in a surface was first investigated by Ueda in \cite{U91} when $t(N_{C/S})\in \mathbb{C}^*\setminus {\rm U}(1)$, where ${\rm U}(1):=\{t\in\mathbb{C}^*\mid |t|=1\}$. 
In \cite{K6}, we slightly generalized his results to the case where, for example, $C$ is a cycle of rational curves \cite[Theorem 1.6]{K6}. 
In that paper, we also treated the case where $t(N_{C/S})\in {\rm U}(1)$, which is the case that $N_{C/S}$ is a ${\rm U}(1)$-flat line bundle: i.e. $N_{C/S}$ admits a $C^\omega$ Hermitian metric with flat curvature. 
In this case, we showed the existence of a pseudoflat neighborhoods system of $C$ under the assumption that $t(N_{C/S})=e^{2\pi\sqrt{-1}\theta}$ holds for a Diophantine irrational number $\theta\in\mathbb{R}$ \cite[Theorem 1.4]{K6}, which can be regarded as an analogue of Ueda's theorem for a non-singular compact curve embedded in a surface \cite[Theorem 3]{U83}. 
Here we remark that Theorem \ref{thm:arnold_for_nodal_rational} is also regarded as an improved version of \cite[Theorem 1.4]{K6}. 

In the proof of Theorem \ref{thm:arnold_for_nodal_rational}, we compare the complex structure of a small neighborhood $V$ of $C$ with that of {\it the standard model} we describe in Example \ref{ex:standard_model} or \ref{ex:standard_model_general}. 
We consider the cohomology class $\alpha=\alpha(C, V)\in H^1(V, \mathcal{O}_V)$ which corresponds to the difference of them. 
Then one can see that it is sufficient to show that the restriction $\alpha|_{V^*}$ of this class to a small neighborhood $V^*$ of $C$ in $V$ is equal to zero in $H^1(V^*, \mathcal{O}_{V^*})$. 
Note that this class satisfies $\alpha|_C=0\in H^1(C, \mathcal{O}_C)$. 
Therefore the problem is reduced to showing the injectivity of the restriction morphism $\lim_{V^*\to} H^1(V^*, \mathcal{O}_{V^*})\to H^1(C, \mathcal{O}_C)$, where $V^*$ runs all the neighborhoods of $C$ in $V$ (Proposition \ref{prop:H^1_main_general}). 
We show this by using a complex dynamical technique originated from \cite{S}, which is also used in the proofs of \cite[Theorem 3]{U83} and \cite[Theorem 1.4]{K6}. 

The main motivation of the present paper comes from \cite{T} and \cite{KK3}. 
In \cite{KK3}, as an application of Arnol'd's theorem \cite{A}, we constructed a K3 surface by holomorphically gluing two open complex surfaces obtained as the complements of tubular neighborhoods of non-singular elliptic curves embedded in the blow-ups of the projective planes at appropriate nine points. 
As described in \cite[\S 4.1.1]{KK3}, this construction can be regarded as a concrete description of a general fiber of a degeneration of K3 surfaces of type II. 
Theorem \ref{thm:arnold_for_nodal_rational} can also be applied to nodal curves embedded in the blow-ups of the projective planes at appropriate nine points (Examples \ref{ex:ABU}, \ref{ex:ABU_2}, and \ref{ex:ABU_3}). 
Toward a concrete description of a general fiber of a degeneration of K3 surfaces of type III, we will investigate these examples precisely. 

The organization of the paper is as follows. 
In \S 2, we correct some fundamental facts on cycles of rational curves. 
Here we also fix coordinates on a neighborhood of each irreducible component of a cycle of rational curves by using Grauert's theorem \cite{G} intrinsically. 
In \S 3, we show the injectivity of the morphism $\lim_{V^*\to} H^1(V^*, \mathcal{O}_{V^*})\to H^1(C, \mathcal{O}_C)$, where $V^*$ runs all the neighborhoods of $C$. 
In \S 4, we prove Theorem \ref{thm:arnold_for_nodal_rational}. 
In \S 5, we investigate Examples \ref{ex:ABU}, \ref{ex:ABU_2}, and \ref{ex:ABU_3} precisely. 
\vskip3mm
{\bf Acknowledgment. } 
The author would like to give heartful thanks to Prof. Tetsuo Ueda whose enormous support and insightful comments were invaluable. 
He thanks Dr. Takahiro Matsushita and Dr. Yuta Nozaki  who gave him many valuable comments on the topological aspects of Levi-flat hypersurfaces which is treated in \S 5. 
He is supported by Leading Initiative for Excellent Young Researchers (No. J171000201).


\section{Preliminaries}

In this section, we collect some fundamental facts and fix some notation on a cycle of rational curves. 

Let $C$ be a cycle of rational curves embedded in a non-singular surface $S$ with Diophantine condition as in Theorem \ref{thm:arnold_for_nodal_rational}. 
Take an open covering $\{U_j\}_j$ of $C$, 
and a small neighborhood $V_j$ of $U_j$ in $S$ with $V_j\cap C=U_j$ for each $j$. 
Denote by $V$ the neighborhood $\bigcup_jV_j$ of $C$. 

It follows from \cite[Proposition 2.5 (2)]{K6} that the pair $(C, S)$ is of infinite type in the sense of \cite{K6}. 
Therefore, from \cite[Theorem 1.4]{K6}, we have that there exists a defining function $w_j$ of $U_j$ in $V_j$ for each $j$ such that $w_j=t_{jk}w_k$ holds for some $t_{jk}\in {\rm U}(1)$ on each $V_{jk}:=V_j\cap V_k$ when $\{U_j\}$ and $\{V_j\}$ are sufficiently fine. 

\subsection{Preliminaries on a rational curve with a node}
\subsubsection{Notation}\label{section:prelim_notation}
Let $C$ be a rational curve with a node. 
In this case, we choose open coverings $\{U_j\}_j$ and $\{V_j\}_j$ such that the index set is $\{0, 1\}$ as follows: 
Let $U_0$ be a small neighborhood of the nodal point of $C$ 
and $U_1$ be the regular part $C_{\rm reg}:=C\setminus \{\text{nodal point}\}$ of $C$. 
By taking $V_j$ as a sufficiently small neighborhood of $U_j$, we may assume that $V_0\cap V_1$ consists of two connected component $V^+$ and $V^-$. 
Let $t_{\pm}$ be elements of ${\rm U}(1)$ such that 
\[
w_1= \begin{cases}
    t_+\cdot w_0 & (\text{on}\ V^+) \\
    t_-\cdot w_0 & (\text{on}\ V^-). 
  \end{cases}
\]
Note that $t:=t_+/t_-=t(N_{C/S})\in {\rm U}(1)\subset \mathbb{C}^*=H^1(C, \mathbb{C}^*)$, see \S \ref{section:Pic}. 

Let $z$ be a non-homogeneous coordinate of the normalization $\widetilde{C}\cong \mathbb{P}^1$ of $C$ such that the preimage of the nodal point is $\{0, \infty\}$. 
As we will see in \S \ref{section:V_tilde}, we can extend the function $z|_{U_1}$ to $V_1$, where we are identifying $\widetilde{C}\setminus\{0, \infty\}$ with $U_1$ (see also \cite{Siu}). 
The resulting holomorphic function on $V_1$ is also denoted by $z$. 
Take coordinates $(x, y)$ of $V_0$ such that $x\cdot y$ is a defining function of $U_0$ in $V_0$. 
These functions $(x, y)$ will also be chosen by more careful argument in \S \ref{section:V_tilde} in actual. 
Denote by $U^+_0$ the subset $\{(x, y)\in V_0\mid y=0\}$ and 
$U^-_0$ the subset $\{(x, y)\in V_0\mid x=0\}$. 
We may assume that $U^+:=V^+\cap U_0$ coincides with $U^+_0\setminus\{\text{nodal point}\}$, and that $U^-:=V^-\cap U_0$ coincides with $U^-_0\setminus\{\text{nodal point}\}$. 

\subsubsection{Picard variety and some cohomologies}\label{section:Pic}

Let $L\in {\rm Pic}^0(C)$ be a topologically trivial line bundle on $C$. 
Then there is a uniquely determined complex constant $t=t(L)\in \mathbb{C}^*$ with 
\[
L=[\{(U^+, t), (U^-, 1)\}]\in \check{H}^1(\{U_j\}, \mathcal{O}_C^*)
=H^1(C, \mathcal{O}_C^*)
\]
where we are using the notation in the previous section (see the arguments around \cite[Lemma 1]{U91}). 
In particular, it is observed that $L$ admits $\mathbb{C}^*$-flat structure: i.e. $L$ admits a flat connection. 
From this fact, one have that ${\rm Pic}^0(C)$ is naturally identified with $H^1(C, \mathbb{C}^*)=\mathbb{C}^*$. 

When $t(L)\in {\rm U}(1)\setminus \{1\}$, $L$ is a non-trivial ${\rm U}(1)$-flat line bundle. 
In this case, one can obtain by considering the long exact sequence comes from $0\to \mathcal{O}_C(L)\to i_*\mathcal{O}_{\widetilde{C}}(i^*L)\to \mathcal{O}_{\{\text{the nodal point}\}}\to 0$ that $H^0(C, L)=H^1(C, L)=0$, where $i\colon\widetilde{C}\to C$ is the normalization 
(see \cite[p. 852]{K6}). 
By the same argument, one also have that $H^0(C, \mathcal{O}_C)\cong H^1(C, \mathcal{O}_C)\cong\mathbb{C}$. 


\subsubsection{Standard model of a neighborhood of a rational curve with a node and some examples}

The following example can be regarded as the standard model of a neighborhood of a rational curve with a node. 

\begin{example}\label{ex:standard_model}
Let $\widetilde{V}$ be a neighborhood of the zero section $\widetilde{C}$ of the line bundle $\pi\colon\mathcal{O}_{\mathbb{P}^1}(-2)\to \mathbb{P}^1$. 
Let $S$ be a non-homogeneous coordinate of $\mathbb{P}^1$. 
We also use the non-homogeneous $T:=S^{-1}$ especially when we observe a neighborhood of $\{S=\infty\}$. 
Let $\xi_0$ and $\xi_\infty$ be fiber coordinates of $\mathcal{O}_{\mathbb{P}^1}(-2)$ defined in a neighborhood of $\{S=0\}$ and $\{T=0\}$, respectively. 
We may assume that $\xi_\infty=\xi_0\cdot S^2$ holds in the fibers over $\mathbb{P}^1\setminus\{S=0, \infty\}$. 

Fix a constant $0<\varepsilon <1$ and let us consider subsets 
\[
\widetilde{V}_0^+:=\{(S, \xi_0)\in \pi^{-1}(\mathbb{C})\mid |S|<\varepsilon,\ |\xi_0|<\varepsilon\}
\]
and 
\[
\widetilde{V}_0^-:=\{(T, \xi_\infty)\in \pi^{-1}(\mathbb{P}^1\setminus\{0\})\mid |T|<\varepsilon,\ |\xi_\infty|<\varepsilon\}
\]
of $\widetilde{V}$. 
By shrinking $\widetilde{V}$ if necessary, we may assume that 
$\pi^{-1}(\{|S|<\varepsilon\})\cap \widetilde{V}=\widetilde{V}_0^+$ and 
$\pi^{-1}(\{|T|<\varepsilon\})\cap \widetilde{V}=\widetilde{V}_0^-$. 
Define a biholmorphism $F\colon \widetilde{V}_0^+\to \widetilde{V}_0^-$ by $F^*(T, \xi_\infty):=(t\cdot \xi_0, S)$ 
(i.e. $F^*T:=T\circ F:=t\cdot \xi_0$ and $F^*\xi_\infty:=\xi_\infty\circ F:=S$), where $t\in {\rm U}(1)$ is a constant. 
Denote by $i\colon \widetilde{V}\to V$ the quotient by the relation induced by $F$. 
Then $V$ is a non-singular surface and the compact analytic subset $C:=i(\widetilde{C})$ is a rational curve with a node such that $t(N_{C/S})=t$. 
\end{example}

Next example is an analogue of Arnol'd--Ueda--Brunella's example \cite{A} \cite{U83} \cite{B}. 

\begin{example}\label{ex:ABU}
Take a plane cubic $C_0\subset \mathbb{P}^2$ which admits only one nodal point, 
and nine points $Z\subset \{p_1, p_2, \dots, p_9\}\subset (C_0)_{\rm reg}$, where $(C_0)_{\rm reg}$ is the non-singular locus of $C_0$. 
Denote by $\pi\colon S\to \mathbb{P}^2$ the blow-up at $Z$ and by $C$ the strict transform $(\pi^{-1})_*C_0$. 
Then it is known that, by taking a normalization $i\colon \mathbb{P}^1\to C_0$ with $i^{-1}((C_0)_{\rm sing})=\{0, \infty\}$ appropriately ($(C_0)_{\rm sing}:=C_0\setminus (C_0)_{\rm reg}$), the complex constant $t=t(N_{C/S})\in\mathbb{C}^*$ can be calculated by 
$t = \prod_{\nu=1}^9i^{-1}(p_\nu) \in \mathbb{C}^* = H^1(C_0, \mathbb{C}^*)$, 
where we are identifying $C_0$ and $C$ via $\pi$. 
Especially, each point of ${\rm Pic}^0(C_0)$ is attained by choosing appropriate nine points configuration $Z$. 
\end{example}

Finally, we give a counter example of Theorem \ref{thm:arnold_for_nodal_rational} when $N_{C/S}$ does not satisfy Diophantine condition. 

\begin{example}
Let $\{(\widetilde{V}, \widetilde{V}_0^\pm, S, T, \xi_0, \xi_\infty)\}$ be those in Example \ref{ex:standard_model}. 
Denote by $\widetilde{W}_0^+$ the subset $\{(S, \xi_0)\in \pi^{-1}(\mathbb{C})\cap\widetilde{V}\mid |S|<1\}$ and 
by $\widetilde{W}_0^-$ the subset $\{(T, \xi_\infty)\in \pi^{-1}(\mathbb{P}^1\setminus\{0\})\cap\widetilde{V}\mid |T|<1\}$ of $\mathcal{O}_{\mathbb{P}^1}(-2)$. 
Note that $\widetilde{V}_0^+\subset \widetilde{W}_0^+$ and $\widetilde{V}_0^-\subset \widetilde{W}_0^-$. 
For sufficiently small positive constant $\delta$, set 
$\widetilde{W}_1:=\{(S, \xi_0)\in \widetilde{V}\mid 1/2<|S|<2,\ |\xi_0|<\delta\}$. 
We may assume that 
$\widetilde{V}_0^+\cap \widetilde{W}_1=\emptyset$ and $\widetilde{V}_0^-\cap \widetilde{W}_1=\emptyset$ hold. 
Take a univalent holomorphic function $\varphi$ defined on $\{w\in\mathbb{C}\mid |w|<\delta\}$ such that $\varphi(0)=0$ and $\lambda:=\varphi'(0)\in{\rm U}(1)$ hold. 
Denote by $\Phi_+\colon\widetilde{W}_1\cap \widetilde{W}_0^+\to \widetilde{W}_0^+$ the map defined by $(\Phi_+)^*(S, \xi_0)=(S, \varphi(S\cdot \xi_0)\cdot S^{-1})$ and by $\Phi_-\colon\widetilde{W}_1\cap \widetilde{W}_0^-\to \widetilde{W}_0^-$ the natural injection. 
Define a surface $W$ by $W:=(\widetilde{W}_0^+\amalg \widetilde{W}_1\amalg \widetilde{W}_0^-)/\sim$, where $\sim$ is the relation generated by 
\[
\begin{cases}
\widetilde{W}_0^+\ni p\sim F(p) \in \widetilde{W}_0^- & \text{if}\ p\in\widetilde{V}_0^+\\
\widetilde{W}_1\ni p\sim \Phi_+(p) \in \widetilde{W}_0^+ & \text{if}\ p\in\left\{(S, \xi_0)\in \widetilde{W}_1\left| \frac{1}{2}<|S|<1\right.\right\}\\
\widetilde{W}_1\ni p\sim \Phi_-(p)\in \widetilde{W}_0^-& \text{if}\ p\in\{(S, \xi_0)\in \widetilde{W}_1\mid 1<|S|<2\}, 
\end{cases}
\]
where $F$ is the one in Example \ref{ex:standard_model} with $t=1$. 
Denote by $C$ the image of $\widetilde{C}$ by the quotient map. 
Note that $C$ is a compact leaf of the holomorphic foliation $\mathcal{F}$ on $W$ whose leaves are defined by 
\[
\begin{cases}
\{S\cdot \xi_0=\text{constant}\} & (\text{on}\ \widetilde{W}_0^+) \\
\{S\cdot \xi_0=\text{constant}\} & (\text{on}\ \widetilde{W}_1) \\
\{T\cdot \xi_\infty=\text{constant}\} & (\text{on}\ \widetilde{W}_0^-). 
\end{cases}
\]
Assume that $\varphi$ is the one as in \cite[p. 606]{U83}. 
Then $t(N_{C/S})=\varphi'(0)$ is a non-torsion element of ${\rm U}(1)$, and 
any small neighborhood $W^*\subset W$ of $C$ includes a compact leaf of $\mathcal{F}$ which is biholomorphic to an elliptic curve and has no intersection with $C$. 
As it follows from the same argument as in \cite[\S 5.3]{U83} that there is no compact subvariery $W^*\setminus C$ for sufficiently small $W^*$ if $C$ admits pseudoflat neighborhoods system, we have that $C$ does not admit a neighborhood as in Example \ref{ex:standard_model} in this example. 
%
\end{example}

\subsection{Definition of the covering map $\widetilde{V}\to V$ and outline of the proof of Theorem \ref{thm:arnold_for_nodal_rational}}\label{section:V_tilde}
Here we use the notation in \S \ref{section:prelim_notation}. 
Take a copy $\widetilde{V}_1$ of $V_1$ and two copies $\widetilde{V}_0^+$ and $\widetilde{V}_0^-$ of $V_0$. 
Denote by $\widetilde{V}$ the manifold constructed by patching $\widetilde{V}_0^+$, $\widetilde{V}_1$ and $\widetilde{V}_0^-$ by considering the natural injections
\[
\xymatrix{
\widetilde{V}_0^+& &\widetilde{V}_1 & & \widetilde{V}_0^- \\
& V^+\ar[ur]\ar[ul]& & V^-\ar[ur]\ar[ul] &\\
}
\]
of $V^\pm$. 
Note that $\widetilde{V}$ can be regarded as an open submanifold of the universal covering of $V$. 
Denote by $i\colon \widetilde{V}\to V$ the natural map. 
In what follows, we regard $\widetilde{V}^\pm_0$ and $\widetilde{V}_1$ as subsets of $\widetilde{V}$. 
Then $i|_{\widetilde{C}}$ is a normalization of $C$, where $\widetilde{C}\subset i^{-1}(C)$ is the irreducible component which is compact. 
By identifying $\widetilde{C}$ and $\mathbb{P}^1$, we may assume that the preimage of the nodal point is $\{0, \infty\}$. 
Denote by $D_0$ and $D_\infty$ the other two irreducible components of $i^{-1}(C)$ which intersects $\widetilde{C}$ at $0$ and $\infty$, respectively. 
Define the defining function $\widetilde{w}$ of the divisor $i^*C=\widetilde{C}+D_0+D_\infty$ of $\widetilde{V}$ by 
\[
\widetilde{w}:=
\begin{cases}
     (\widetilde{V}_0^+\to V_0)^*(t_+\cdot w_0) & (\text{on}\ \widetilde{V}_0^+) \\
     (\widetilde{V}_1\to V_1)^*w_1 & (\text{on}\ \widetilde{V}_1) \\
     (\widetilde{V}_0^+\to V_0)^*(t_-\cdot w_0) & (\text{on}\ \widetilde{V}_0^-), 
  \end{cases}
\]
where $\widetilde{V}_0^+\to V_0$, $\widetilde{V}_1\to V_1$ and $\widetilde{V}_0^+\to V_0$ be the natural biholomorphisms. 
By a simple argument, we have that ${\rm deg}\,N_{\widetilde{C}/\widetilde{V}}=-2$. 
Therefore, it follows from Grauert's theorem \cite{G} that $\widetilde{V}$ can be holomorphically embedded in the total space of the line bundle $\mathcal{O}_{\mathbb{P}^1}(-2)\to \mathbb{P}^1$ by shrinking $\widetilde{V}$ (see also \cite[Theorem 2.5.2]{CM}). 
In what follows, we regard $\widetilde{V}$ as a subset of $\mathcal{O}_{\mathbb{P}^1}(-2)$ and identify $\widetilde{C}$ with the zero section via this embedding. 

Take a strictly pseudoconvex neighborhood $\mathcal{V}$ of $\widetilde{C}$ in $\widetilde{V}$. 
It follows from Ohsawa's vanishing theorem \cite[Theorem 4.5]{O} that $H^1(\mathcal{V}, \mathcal{O}_{\mathcal{V}})=0$. 
Thus we have that the line bundle on $\widetilde{V}$ corresponds to the divisor $D_0-D_\infty$ is holomorphically trivial. 
Therefore there exists a holomorphic map $p\colon \mathcal{V}\to \mathbb{P}^1$ such that 
$p^*(\{0\}-\{\infty\})=D_0-D_\infty$ holds as divisors. 
By shrinking $\widetilde{V}$ so that $\widetilde{V}\subset \mathcal{V}$, we may assume that the map $p$ is defined on $\widetilde{V}$. 

Let $S=T^{-1}$ be non-homogeneous coordinate of $\mathbb{P}^1$. 
Denote also by $S$ and $T$ the meromorphic functions $p^*S$ and $p^*T$ on $\widetilde{V}$, respectively. 
Then we have that $D_0=\{S=0\}=\{T=\infty\}$ and $D_\infty=\{S=\infty\}=\{T=0\}$. 
Setting $\xi_0:=\widetilde{w}\cdot S^{-1}$ on a neighborhood of $D_0$ and 
$\xi_\infty:=\widetilde{w}\cdot T^{-1}$ on a neighborhood of $D_\infty$, 
we regard $(S, \xi_0)$ and $(T, \xi_\infty)$ as coordinates 
of a neighborhood of $D_0$ and $D_\infty$, respectively. 
Denote by $F\colon \widetilde{V}^+_0\to \widetilde{V}^-_0$ the biholomorphism such that $i^{-1}(i(p))=\{p, F(p)\}$ for each $p\in \widetilde{V}^+_0$. 
As it hold that $F^*\widetilde{w}=t\cdot \widetilde{w}$, 
$F^*D_\infty=\widetilde{C}\cap \widetilde{V}^+_0$ and 
that $F^*(\widetilde{C}\cap \widetilde{V}^-_0)=D_0$, 
we have that 
\[
F^*(T, \xi_\infty)=
\left(\frac{t\cdot \xi_0}{G(S, \xi_0)},\ G(S, \xi_0)\cdot S\right)
\]
holds for a nowhere vanishing holomorphic function $G$ defined on $\widetilde{V}^+_0$, where $t=t_+/t_-$. 
By changing the scaling of $\widetilde{w}$, we may assume that $G(0, 0)=1$. 
In \S \ref{section:prf}, we will prove Theorem \ref{thm:arnold_for_nodal_rational} by showing that one may assume that $G\equiv 1$ by changing coordinate functions appropriately. 

\subsection{Preliminaries on a cycle of rational curves in general}\label{prelim:general}

Let $C$ be a cycle of rational curves in general. 
Denote by $n=n(C)$ the number of the irreducible components of $C$. 
Here we treat the case where $n\geq 2$. 
Denote by $\{C_{(\nu)}\}_{\nu=1}^n$ the set of all irreducible components of $C$. 
We sometimes use the notation $C_{(0)}:=C_{(n)}$. 
We may assume that $C_{(\nu)}\cap C_{(\mu)}\not=\emptyset$ if and only if $\nu-\mu=\pm 1$ modulo $n$. 
It holds that $H^1(C, \mathcal{O}_C)=\mathbb{C}$ also in this case, 
since $H^1(C, i_*\mathcal{O}_{\widetilde{C}})=H^1(\widetilde{C}, \mathcal{O}_{\widetilde{C}})=0$ follows from the same exact sequence as we considered in \S \ref{section:Pic}, where $i\colon\widetilde{C}\to C$ is the normalization 
(Note that the higher direct images vanish for $i$, since $i$ is a finite morphism). 
Thus ${\rm Pic}^0(C)$ is naturally identified with $H^1(C, \mathbb{C}^*)=\mathbb{C}^*$ also in this case. 

The following example can be regarded as the standard model of a neighborhood of $C$. 

\begin{example}\label{ex:standard_model_general}
Let $\{(\widetilde{V}_{(\nu)}, \widetilde{C}_{(\nu)}, \widetilde{V}_{0 (\nu)}^\pm, S_{(\nu)}, T_{(\nu)}, \xi_{0 (\nu)}, \xi_{\infty (\nu)})\}_{\nu=1}^n$ be $n$-copies of \linebreak
$(\widetilde{V}, \widetilde{C}, \widetilde{V}_{0}^\pm, S, T, \xi_{0}, \xi_{\infty})$ in Example \ref{ex:standard_model}. 
Define a biholomorphism $F_{\nu+1, \nu}\colon \widetilde{V}_{0 (\nu+1)}^+\to \widetilde{V}_{0 (\nu)}^-$ by 
$(F_{\nu+1, \nu})^*(T_{(\nu)}, \xi_{\infty (\nu)})=(t_{\nu+1, \nu}\cdot \xi_{0 (\nu)}, S_{(\nu)})$ 
for $\nu=0, 1, \dots, n-1$, where $t_{\nu+1, \nu}\in {\rm U}(1)$ and 
\[
(\widetilde{V}_{(0)}, \widetilde{C}_{(0)}, \widetilde{V}_{0 (0)}^\pm, S_{(0)}, T_{(0)}, \xi_{0 (0)}, \xi_{\infty (0)}):=(\widetilde{V}_{(n)}, \widetilde{C}_{(n)}, \widetilde{V}_{0 (n)}^\pm, S_{(n)}, T_{(n)}, \xi_{0 (n)}, \xi_{\infty (n)}). 
\]
Let $i\colon \widetilde{V}:=\coprod_{\nu=1}^n\widetilde{V}_\nu\to V$ be the quotient by the relation induced by $F_{\nu+1, \nu}$'s. 
Denote by $C$ the image $i(\widetilde{C})$, where $\widetilde{C}:=\coprod_{\nu=1}^n\widetilde{C}_{(\nu)}$. 
Then $C$ is a cycle of $n$ rational curves embedded in $V$ with $t(N_{C/V})=\prod_{\nu=0}^{n-1}t_{\nu+1, \nu}$. 
\end{example}

\begin{remark}\label{rmk:V_tilde_general}
It follows from the same argument as in \S \ref{section:V_tilde} that 
one can construct a finite covering map $i\colon \widetilde{V}\to V$ as in Example \ref{ex:standard_model_general} for a small neighborhood $V$ of $C$ also in the case where $C$ consists of $n$ irreducible components ($n\geq 2$). 
In this case, $\widetilde{V}$ is the disjoint union of a neighborhood $\widetilde{V}_{(\nu)}$ of each irreducible component $C_{(\nu)}$ of $C$ with the same local coordinates $(S_{(\nu)}, T_{(\nu)}, \xi_{0 (\nu)}, \xi_{\infty (\nu)})$ as in Example \ref{ex:standard_model_general} (Here we use Grauert's theorem \cite{G} again). 
In general, the gluing morphism $F_{\nu+1, \nu}\colon \widetilde{V}_{0 (\nu+1)}^+\to \widetilde{V}_{0 (\nu)}^-$ needs not to coincides with the one in Example \ref{ex:standard_model_general}. 
From the same argument as in \S \ref{section:V_tilde} by using \cite[Theorem 1.4]{K6}, 
it follows that, by choosing $S_{(\nu)}, T_{(\nu)}, \xi_{0 (\nu)}$ and $\xi_{\infty (\nu)}$ suitably, we may assume that 
\[
(F_{\nu+1, \nu})^*(T_{(\nu)}, \xi_{\infty (\nu)})=\left(\frac{t_{\nu+1, \nu}\cdot \xi_{0 (\nu)}}{G_{\nu+1, \nu}(S_{(\nu)}, \xi_{0 (\nu)})},\ G_{\nu+1, \nu}(S_{(\nu)}, \xi_{0 (\nu)})\cdot S_{(\nu)}\right)
\]
holds for a nowhere vanishing holomorphic function $G_{\nu+1, \nu}$ defined on $\widetilde{V}_{0 (\nu+1)}^+$ and a constant $t_{\nu+1, \nu}\in{\rm U}(1)$. 
In \S \ref{section:prf}, 
we will prove Theorem \ref{thm:arnold_for_nodal_rational} by showing that one may assume that $G_{\nu+1, \nu}\equiv 1$ by changing the coordinate functions appropriately. 
\end{remark}

\begin{example}\label{ex:ABU_2}
Fix a plane cubic $C_0\subset \mathbb{P}^2$ which admits only nodal singularities and consists of two irreducible components, say $C_0^{(1)}$ and $C_0^{(2)}$. 
One may assume that $C_0^{(\nu)}$ is of degree $\nu$ for $\nu=1, 2$. 
Take three points $\{p_1, p_2, p_3\}\subset C_0^{(1)}\cap (C_0)_{\rm reg}$ and 
six points $\{p_4, p_5, \dots, p_9\}\subset C_0^{(2)}\cap (C_0)_{\rm reg}$. 
Denote by $\pi\colon S\to \mathbb{P}^2$ the blow-up at $Z:=\{p_1, p_2, \dots, p_9\}$ and by $C$ the strict transform $(\pi^{-1})_*C_0$. 
Then it is known that, by taking a normalization $i\colon \mathbb{P}^1\amalg\mathbb{P}^1\to C_0$ with $i^{-1}((C_0)_{\rm sing})= \{0, \infty\}$ appropriately, the complex constant $t=t(N_{C/S})\in\mathbb{C}^*$ can be calculated by 
$t = \prod_{\nu=1}^9i^{-1}(p_\nu) \in \mathbb{C}^* = H^1(C_0, \mathbb{C}^*)$, 
where we are identifying $C_0$ and $C$ via $\pi$. 
\end{example}

\begin{example}\label{ex:ABU_3}
Fix a plane cubic $C_0\subset \mathbb{P}^2$ which admits only nodal singularities and consists of three irreducible components, say $C_0^{(1)}$, $C_0^{(2)}$ and $C_0^{(3)}$. 
Each $C_0^{(\nu)}$ is a line for $\nu=1, 2, 3$. 
Take three points $\{p_1, p_2, p_3\}\subset C_0^{(1)}\cap (C_0)_{\rm reg}$, 
$\{p_4, p_5, p_6\}\subset C_0^{(2)}\cap (C_0)_{\rm reg}$, and 
$\{p_7, p_8, p_9\}\subset C_0^{(3)}\cap (C_0)_{\rm reg}$. 
Denote by $\pi\colon S\to \mathbb{P}^2$ the blow-up at $Z:=\{p_1, p_2, \dots, p_9\}$ and by $C$ the strict transform $(\pi^{-1})_*C_0$. 
Then it is known that, by taking a normalization $i\colon \mathbb{P}^1\amalg\mathbb{P}^1\amalg\mathbb{P}^1\to C_0$ with $i^{-1}((C_0)_{\rm sing})= \{0, \infty\}$ appropriately, the complex constant $t=t(N_{C/S})\in\mathbb{C}^*$ can be calculated by 
$t = \prod_{\nu=1}^9i^{-1}(p_\nu) \in \mathbb{C}^* = H^1(C_0, \mathbb{C}^*)$, 
where we are identifying $C_0$ and $C$ via $\pi$. 
\end{example}

Note that each point of ${\rm Pic}^0(C_0)$ is attained by choosing appropriate nine points configuration $Z$ in Examples \ref{ex:ABU_2} and \ref{ex:ABU_3} (as in Example \ref{ex:ABU}). 

%
%
%
%
%
%


\section{Injectivity of the restriction $\lim_{V^*\to} H^1(V^*, \mathcal{O}_{V^*})\to H^1(C, \mathcal{O}_C)$}

We will show Theorem \ref{thm:arnold_for_nodal_rational} in \S \ref{section:prf} by the strategy as we mentioned in \S \ref{section:V_tilde} and Remark \ref{rmk:V_tilde_general}. 
When $C$ is a rational curve with a node, for example, we will choose suitable coordinates of $\widetilde{V}$ so that $G\equiv 1$ holds. 
Consider the composition $g$ of the natural biholomorphism $V_0\to \widetilde{V}^+_0$ and the branch of $\frac{1}{2\pi\sqrt{-1}}\log G$ such that $g(0, 0)=0$. 
By the arguments we will explain the details in \S \ref{section:prf}, the problem can be reduced to showing that the cohomology class 
$\alpha:=[\{(V^+, -g|_{V^+}), (V^-, 0)\}]
\in \check{H}^1(\{V_j\}, \mathcal{O}_V)$ 
is trivial. 
As it is easily observed that $\alpha|_C=0\in H^1(C, \mathcal{O}_C)$ (see the proof of ``Proposition \ref{prop:H^1} $\Rightarrow$ Proposition \ref{prop:H^1_main}" below), it is sufficient to show the injectivity of the restriction $H^1(V, \mathcal{O}_V)\to H^1(C, \mathcal{O}_C)$ by shrinking $V$ in a suitable sense. 
For such a purpose, we will show the following: 

\begin{proposition}\label{prop:H^1_main_general}
Let $C$ be a cycle of a curve embedded in non-singular surface $V$ such that the normal bundle $N_{C/V}$ is topologically trivial and satisfies Diophantine condition as in Theorem \ref{thm:arnold_for_nodal_rational}. 
For any element $\alpha$ of the kernel of the restriction $H^1(V, \mathcal{O}_V)\to H^1(C, \mathcal{O}_C)$, there exists a neighborhood $V^*$ of $C$ such that $\alpha|_{V^*}=0\in H^1(V^*, \mathcal{O}_{V^*})$. 
\end{proposition}


\subsection{Proof of Proposition \ref{prop:H^1_main_general} when $C$ is a rational curve with a node}

\subsubsection{Notation and statement in this case}

Assume that $C$ is a rational curve with a node. 
Then we can use the notation as in \S \ref{section:V_tilde}. 
By a simple argument, Proposition \ref{prop:H^1_main_general} can be reworded as follows: 

\begin{proposition}\label{prop:H^1_main}
Let $F_+$ and $F_-$ be holomorphic functions defined on $V^+$ and $V^-$, respectively, such that $\{(U^\pm, F_\pm|_{U^\pm})\}$ extends to a holomorphic function defined on $U_0$. 
Then there exists a neighborhood $V^*$ of $C$ such that the class 
\[
\alpha:=
[\{(V^*\cap V^+, F_+|_{V^*\cap V^+}), (V^*\cap V^-, F_-|_{V^*\cap V^-})\}]\in \check{H}^1(\{V^*\cap V_j\}, \mathcal{O}_{V^*})
\]
is trivial. 
\end{proposition}

As will be proven immediately after, Proposition \ref{prop:H^1_main} follows from: 

\begin{proposition}\label{prop:H^1}
Let $F_+$ and $F_-$ be holomorphic functions defined on $V^+$ and $V^-$, respectively, such that $F_{\pm}|_{U^\pm}\equiv 0$. 
Then there exists a neighborhood $V^*\subset V$ of $C$ such that the \v{C}ech cohomology class
\[
[\{(V^*\cap V^+, F_+|_{V^*\cap V^+}), (V^*\cap V^-, F_-|_{V^*\cap V^-})\}]\in \check{H}^1(\{V^*\cap V_j\}, \mathcal{O}_{V^*}(-C))
\]
is trivial. 
\end{proposition}

\begin{proof}[Proof of ``Proposition \ref{prop:H^1} $\Rightarrow$ Proposition \ref{prop:H^1_main}"]
Denote by $g_0$ the extension of $\{(U^\pm, F_\pm|_{U^\pm})\}$ to $U_0$. 
As $V_0$ is coverd by a Stein neighborhood of $U_0$, we obtain a holomorphic function $G_0$ on $V_0$ such that $G_0|_{U_0}=g_0$. 
By using a function $G_1$ on $V_1$ defined by $G_1\equiv 0$, consider 
\[
\beta:=[\{(V^+, (G_0-G_1)|_{V^+}), (V^-, (G_0-G_1)|_{V^-})\}]\in \check{H}^1(\{V_j\}, \mathcal{O}_{V}). 
\]
Then it follows from Proposition \ref{prop:H^1} that the class 
$(\alpha-\beta)\in \check{H}^1(\{V^*\cap V_j\}, \mathcal{O}_{V^*}(-C))$ 
is trivial for a neighborhood $V^*$ of $C$, which proves Proposition \ref{prop:H^1_main}. 
\end{proof}

Here we first give some notation which will be used in the proof of Proposition \ref{prop:H^1}. 
Let $V_0^*$ be a small neighborhood of the nodal point in $V_0$. 
Denote by $x$ the holomorphic function obtained by pulling back the function $S$ by the natural biholomorphism $V_0\to \widetilde{V}_0^+$, 
and by $y$ the one obtainded by pulling back the function $T$ by the natural biholomorphism $V_0\to \widetilde{V}_0^-$. 
We regard $(x, y)$ as coordinates of a neighborhood of $\overline{V_0^*}$. 
Note that $x\cdot y$ is a local defining function of $C$ in this locus. 
For sufficiently small positive constants $\varepsilon$ and $\delta$, we may assume that 
$V^*_0:=\{(x, y)\in V_0\mid \max\{|x|,\ |y|\}<2\varepsilon,\ |w_0|<\delta\}$. 
Denote by $U^*_0$ the subset $V^*_0\cap C$: i.e. 
$U^*_0=\{(x, y)\in V^*_0\mid |x|<2\varepsilon,\ y=0\}\cup\{(x, y)\in V^*_0\mid x=0,\ |y|<2\varepsilon\}$. 
In what follows we always assume that $\varepsilon$ and $\delta$ are sufficiently small so that $V_0^*$ is a relatively compact subset of $V_0$. 

Next we give a definition of a relatively compact subset $V_1^*$ of $V_1$. 
Denote by $z$ the holomorphic function obtained by pulling back the function $S$ by the natural biholomorphism $V_1\to \widetilde{V}_1$. 
Denote by $V_1^*$ the subset $\{(z, w_1)\in V_1\mid \varepsilon<|z|<1/\varepsilon,\ |w_1|<\delta\}$, where we are regarding $(z, w_1)$ as coordinates of this locus. 
Let $U^*_1$ be the subset of $U_1$ defined by $U^*_1:=V^*_1\cap C$: i.e. 
$U^*_1=\{(z, w_1)\in V^*_1\mid \varepsilon<|z|<1/\varepsilon,\ w_1=0\}$. 
Set $U^*_+:=U^+_0\cap U^*_1=\{(x, y)\in V^*_0\mid \varepsilon<|x|<2\varepsilon,\ y=0\}$ 
and 
$U^*_-:=U^-_0\cap U^*_1=\{(x, y)\in V^*_0\mid x=0,\ \varepsilon<|y|<2\varepsilon\}$. 
Denote by $V^*_\pm$ the connected components of $V^*_0\cap V^*_1$ which includes $U^*_\pm$ respectively, 
and by $V^*$ the subset $V^*_0\cup V^*_1=i(\{|\widetilde{w}|<\delta\})$. 

In what follows, we fix $\varepsilon$ and do not vary this value any more, 
whereas we will shrink $\delta$ as necessary. 

\subsubsection{Outline of the proof of Proposition \ref{prop:H^1}}

We will construct holomorphic functions $F_j$ on $V^*_j$ for each $j=0, 1$ such that 
$F_0|_{V^*_\pm}-F_1|_{V^*_\pm}=F_\pm|_{V^*_\pm}$ 
holds on each $V^*_\pm$ by shrinking $\delta$. 
Actually, it is sufficient to construct such $\{(V_j^*, F_j)\}$, 
since we can construct $\widehat{F}_j\colon V^*\cap V_j\to \mathbb{C}$ such that 
$\delta\{(V^*\cap V_j, \widehat{F}_j)\}
=\{(V^*\cap V^\pm, F_\pm|_{V^*\cap V^\pm})\}$ from them as follows: 
Set $\widehat{F}_j(p):=F_j(p)$ for $p\in V^*_j$. 
For $p\in (V^*\cap V_j)\setminus V^*_j$, set 
\[
\widehat{F}_j(p):= \begin{cases}
     F_{1-j}(p)+(-1)^j\cdot F_+(p) & (\text{if}\ p\in V^+) \\
     F_{1-j}(p)+(-1)^j\cdot F_-(p) & (\text{if}\ p\in V^-). 
  \end{cases}
\]
Note that $p\in V^*_{1-j}$, and thus it holds that $p\in V_j\cap V^*_{1-j}\subset V_0\cap V_1=V^+\cup V^-$ in this case. 

\subsubsection{Proof of Proposition \ref{prop:H^1} (Step 1: Construction of $F_j$'s as formal power series)}

In this step, we will construct $F_j$'s in the form of 
\[
F_0(x, y)=\sum_{\nu=1}^\infty a_{0, \nu}(x, y)\cdot w_0^\nu
\]
and 
\[
F_1(z, w_1)=\sum_{\nu=1}^\infty a_{1, \nu}(z)\cdot w_1^\nu 
\]
formally. 
Here $a_{1, \nu}$ is a function defined on $U_1^*$, which is also be regarded as a function on $V_1^*$ by pulling back the natural projection $(z, w_1)\mapsto z$. 
Similarly, $a_{0, \nu}$ is a function defined on $U_0^*$ with 
\[
a_{0, \nu} = \begin{cases}
     p_\nu + r_\nu & (\text{if}\ p\in U^+\cap U^*_0) \\
     q_\nu + r_\nu & (\text{if}\ p\in U^-\cap U^*_0), 
  \end{cases}
\]
where $p_\nu(x)$ is a holomorphic function on $U_+^*$ with $p_\nu(0)=0$,  
$q_\nu(y)$ is a holomorphic function on $U_-^*$ with $q_\nu(0)=0$, 
and $r_\nu\in\mathbb{C}$ is a constant. 
We also regard $a_{0, \nu}$ as a function defined on $V_0^*$ by setting 
$a_{0, \nu}(x, y):=p_\nu(x)+q_\nu(y)+r_\nu$, 
where $p_\nu$ and $q_\nu$ are extended by considering the pull-back by the projection $(x, y)\mapsto x$ and $(x, y)\mapsto y$, respectively. 
Denote by 
\[
F_\pm (z, w_1)=\sum_{\nu=1}^\infty b_{\pm, \nu}(z)\cdot w_1^\nu
\]
the expansion of $F_\pm$ by $w_1$ on $V^*_\pm$. 

First, let us construct $\{a_{j, 1}\}_{j=0, 1}$. 
As $N_{C/S}$ is non-torsion, 
it holds that $\check{H}^1(\{U_j\}, N_{C/S}^{-1})=0$ (see \S \ref{section:Pic}). 
Therefore, by considering the $1$-cocycle 
$[\{(U_\pm^*, b_{\pm, 1})\}]\in \check{H}^1(\{U_j^*\}, N_{C/S}^{-1})$, 
one can take $\{a_{j, 1}\}$ such that 
\[
\begin{cases}
    t_+^{-1}a_{0, 1}(z)-a_{1, 1}(z)=b_{+, 1}(z) & (\text{on}\ U^*_+) \\
    t_-^{-1}a_{0, 1}(z)-a_{1, 1}(z)=b_{-, 1}(z) & (\text{on}\ U^*_-). 
  \end{cases}
\]
Note that such $\{a_{j, 1}\}$ is unique since $H^1(C, N_{C/S}^{-1})=0$. 
By letting $r_1$ be that value of $a_{0, 1}$ at the nodal point, $p_1$ and $q_1$ are uniquely determined. 
We here remark that, for any choice of the other coefficients $\{a_{j, \nu}\}_{j=0, 1, \nu\geq 2}$, we have that 
\[
F_0-F_1= \begin{cases}
     F_++O(w_1^2) & (\text{on}\ V^*_+) \\
     F_-+O(w_1^2) & (\text{on}\ V^*_-)
  \end{cases}
\]
holds as $w_1\to 0$. 

Next, we construct $\{a_{j, n+1}\}$ by assuming that $\{a_{j, \nu}\}_{j=0, 1, \nu\leq n}$ is determined so that the following inductive assumption holds: 
for any choice of $\{a_{j, \nu}\}_{j=0, 1, \nu\geq n+1}$, 
\[
F_0-F_1= \begin{cases}
     F_++O(w_1^{n+1}) & (\text{on}\ V^*_+) \\
     F_-+O(w_1^{n+1}) & (\text{on}\ V^*_-)
  \end{cases}
\]
holds as $w_1\to 0$. 
In what follows, we regard $\{a_{j, \nu}\}_{j=0, 1, \nu\geq n+1}$ as unknown functions. 
Denote by 
\[
p_\nu(x(z, w_1))= \begin{cases}
     p_\nu(x(z)) & (\text{on}\ V^*_+) \\
     \sum_{\lambda=1}^\infty P^-_{\nu, \lambda}(z)\cdot w_1^\lambda & (\text{on}\ V^*_-)
  \end{cases}
\]
and 
\[
q_\nu(y(z, w_1))= \begin{cases}
     \sum_{\lambda=1}^\infty Q^+_{\nu, \lambda}(z)\cdot w_1^\lambda & (\text{on}\ V^*_+) \\
     q_\nu(y(z)) & (\text{on}\ V^*_-)
  \end{cases}
\]
the expansion of $p_\nu$ and $q_\nu$ by $w_1$ respectively 
(Note that $x=x(z, w_1)$ and $y=y(z, w_1)$ do not depend on $w_1$ on $V^*_+$ and $V^*_-$, respectively, in our coordinates, 
and that $q_\nu|_{U^+}\equiv q_\nu(0)=0$ and $p_\nu|_{U^-}\equiv p_\nu(0)=0$). 

On $V_+$, one can expand $F_0|_{V^*_+}$ as follows: 
\[
F_0|_{V^*_+} = \sum_{\nu=1}^\infty a_{0, \nu}(x, y)\cdot w_0^\nu 
= \sum_{\nu=1}^\infty t_+^{-\nu}\cdot \left(
p_\nu(x(z))
+\sum_{\lambda=1}^\infty Q^+_{\nu, \lambda}(z)\cdot w_1^\lambda
+r_\nu
\right)\cdot w_1^\nu. 
\]
By setting 
\[
h^+_m(z):=
\sum_{\nu=1}^{m-1} t_+^{-\nu}\cdot Q^+_{\nu, m-\nu}(z), 
\]
we have that the coefficient of $w_1^m$ in the expansion of $F_0|_{V^*_+}$ is $h^+_m(z)+t_+^{-m} \left(p_m(x(z))+r_m\right)$. 
The function $h^+_m$ can be regarded as a function obtained by pulling back a function on $U^*_+$ by the local projection $(z, w_1)\mapsto z$, which coincides with $(x, y)\mapsto x$ in this locus. 
Note that $\{h^+_m\}_{m\leq n}$ are regarded as known functions, since $h^+_m$ depends only on the data $\{q_\nu\}_{\nu=1}^{m-1}$. 
By a simple observation, it turns out that one should construct $a_{j, n+1}$'s 
so that 
\[
b_{+, n+1}(z)=h^+_{n+1}(z)+t_+^{-n-1}\left(p_{n+1}(x(z))+r_{n+1}\right)-a_{1, n+1}(z)
\]
holds on $U^*_+$ in order for the inductive assumption to hold for $n+1$. 

Similarly, we have that 
\[
F_0|_{V^*_-} = \sum_{\nu=1}^\infty a_{0, \nu}(x, y)\cdot w_0^\nu 
= \sum_{\nu=1}^\infty t_-^{-\nu}\cdot \left(
\sum_{\lambda=1}^\infty P^-_{\nu, \lambda}(z)\cdot w_1^\lambda
+q_\nu(y(z))
+r_\nu
\right)\cdot w_1^\nu 
\]
on $V^*_-$. By setting 
\[
h^-_m(z):=
\sum_{\nu=1}^{m-1} t_-^{-\nu}\cdot P^-_{\nu, m-\nu}(z), 
\]
we have that the coefficient of the expansion of $h^-_m(z)$ in $w_1^m$ is 
$h^-_m(z)+t_-^{-m}\left(q_m(y(z))+r_m\right)$. 
The function $h^-_m$ can be regarded as a function obtained by pulling back a function on $U^*_-$ by the local projection $(z, w_1)\mapsto z$, which coincides with $(x, y)\mapsto y$ in this locus. 
Note that $\{h^-_m\}_{m\leq n}$ are regarded as known functions, since $h^+_m$ depends only on the data $\{p_\nu\}_{\nu=1}^{m-1}$. 
By a simple observation, it turns out that one should construct $a_{j, n+1}$'s 
so that 
\[
b_{-, n+1}(z)=h^-_{n+1}(z)+t_-^{-n-1}\cdot \left(q_{n+1}(y(z))+r_{n+1}\right)-a_{1, n+1}(z)
\]
holds on $U^*_-$ in order for the inductive assumption to hold for $n+1$. 

By the observations above, we have that 
$b_{\pm, n+1}(z)-h^\pm_{n+1}(z)$ is known function after we finish defining $\{a_{j, \nu}\}_{j=0, 1, \nu\leq n}$. 
Therefore, we can define $\{(U^*_0, a_{0, n+1}(x, y)=p_{n+1}(x)+q_{n+1}(y)+r_n), (U^*_1, a_{1, n+1}(z))\}$ by considering the equations 
\[
\begin{cases}
    t_+^{-n-1}a_{0, n+1}(z)-a_{1, n+1}(z)=b_{+, n+1}(z)-h^+_{n+1}(z) & (\text{on}\ U^*_+) \\
    t_-^{-n-1}a_{0, n+1}(z)-a_{1, n+1}(z)=b_{-, n+1}(z)-h^-_{n+1}(z)   & (\text{on}\ U^*_-). 
  \end{cases}
\]
As $H^0(C, N_{C/S}^{-n})=H^1(C, N_{C/S}^{-n})=0$ (see \S \ref{section:Pic}), 
we actually have the unique solution. 
 
\subsubsection{Proof of Proposition \ref{prop:H^1}: (Step 2: Estimate of the coefficient functions)}
As $V^*_\pm\Subset V^\pm$, there exists a constant $M$ such that 
\[
\max\left\{
\sup_{V^*_+}|F_+|,\ 
\sup_{V^*_-}|F_-|
\right\}
<M. 
\]
In what follows, we assume that $M>1$. 
Fix a positive constant $R$ sufficiently larger than $1/\delta, 1/\varepsilon$, $\sup_{V^*_+}|w_0/y|$, $\sup_{V^*_+}|w_1/y|$, $\sup_{V^*_+}|w_1/y|$, $\sup_{V^*_-}|w_0/x|$, $\sup_{V^*_-}|w_1/x|$, and the inverses of these. 
Then we may assume that 
\[
\{(z, w_1)\mid \varepsilon<|z|<2\varepsilon ,\ |w_1|=1/R\}\subset V^*_+
\]
and 
\[
\{(z, w_1)\mid 1/(2\varepsilon)<|z|<1/\varepsilon ,\ |w_1|=1/R\}\subset V^*_-
\]
hold (see also Remark \ref{rmk:R}). 

Let $B(X)=X+\sum_{\nu=2}^\infty B_\nu X^\nu$ be the formal power series defined by 
\begin{equation}\label{eq:func_eq_B}
\sum_{\nu=2}^\infty |1-t^{\nu-1}|\cdot B_\nu X^\nu=KRM\frac{B(X)^2}{1-RB(X)}, 
\end{equation}
where the constant $K$ is a positive real number as in Lemma \ref{lem:ueda_type_estim}. 
Note that it follows from the argument in \cite{S} that $B(X)$ has a positive radius of convergence (see also \cite[Lemma 5]{U83}). 
Define a convergent power series $A(X)=\sum_{\nu=1}^\infty A_\nu X^\nu$ by 
$A_{n}:=B_{n+1}$ ($n\geq 1$): i.e. $B(X)=X+XA(X)$. 
In this step, we show that 
\begin{equation}\label{eq:A_nu}
\max_{j=0, 1}\sup_{p\in U_j^*}|a_{j, \nu}(p)|\leq A_\nu
\end{equation}
holds for each $\nu$ by induction. 

First, by Cauchy's inequality, we have that 
\[
\sup_{z\in U^*_\pm}|b_{\pm, 1}(z)|
\leq 
M\cdot R. 
\]
Therefore, the inequality (\ref{eq:A_nu}) for $\nu=1$ follows from Lemma \ref{lem:ueda_type_estim} below. 

Next we show the inequality (\ref{eq:A_nu}) for $\nu=n+1$ by assuming that it holds for $\nu=1, 2, \dots, n$. 
As it holds that 
$|h^+_{n+1}(z)|\leq
\sum_{\nu=1}^{n} |Q^+_{\nu, n+1-\nu}(z)|$, we have that 
\[
\sup_{U^*_+}|Q^+_{\nu, \lambda}|
\leq A_\nu\cdot R^\lambda
\]
holds by Cauchy's inequality. 
Therefore it follows that 
\[
\sup_{U^*_+}|h^+_{n+1}(z)|\leq
\sum_{\nu=1}^{n} A_\nu\cdot R^{n+1-\nu}
=\text{the coefficient of}\ X^{n+1}\ \text{in the expansion of}\ 
\frac{RXA(X)}{1-RX}. 
\]
Note that the same estimate holds also for $h^-_{n+1}$. As it holds that 
\[
\sup_{z\in U^+\cap U^*_\pm}|b_{\pm, n+1}(z)|
\leq MR^{n+1}
=\text{the coefficient of}\ X^{n+1}\ \text{in the expansion of}\ 
\frac{MRX}{1-RX}, 
\]
it follows from Lemma \ref{lem:ueda_type_estim} that 
\[
\max_{j=0, 1}\sup_{U^*_j}|a_{j, n+1}|\leq 
\text{the coefficient of}\ X^{n+1}\ \text{in the expansion of}\ 
\frac{1}{|1-t^{n+1}|}\cdot \frac{KRX(A(X)+M)}{1-RX}. 
\]
As $M\geq 1$, we have that 
\begin{eqnarray}
&&\text{the coefficient of}\ X^{n+1}\ \text{in the expansion of}\ 
\frac{1}{|1-t^{n+1}|}\cdot \frac{KRX(A(X)+M)}{1-RX}\nonumber \\
&\leq& \text{the coefficient of}\ X^{n+1}\ \text{in the expansion of}\ 
\frac{1}{|1-t^{n+1}|}\cdot \frac{KRMB(X)}{1-RX}\nonumber \\
&=& \text{the coefficient of}\ X^{n+2}\ \text{in the expansion of}\ 
\frac{1}{|1-t^{n+1}|}\cdot \frac{KRMXB(X)}{1-RX}\nonumber \\
&\leq& \text{the coefficient of}\ X^{n+2}\ \text{in the expansion of}\ 
\frac{KRM}{|1-t^{n+1}|}\cdot \frac{B(X)^2}{1-RB(X)}. \nonumber 
\end{eqnarray}

Thus we have the inequality (\ref{eq:A_nu}) for $\nu=n+1$ by the equation (\ref{eq:func_eq_B}). 

\subsubsection{Proof of Proposition \ref{prop:H^1} (Step 3: Convergence of $F_j$'s)}

Let us shrink $\delta$ so that it is smaller than the radius of convergence of the poser series $A(X)$. 
Then it clearly holds that 
$\sup_{V^*_1}|a_{1, \nu}|\leq A_\nu$ when we regard $a_{1, \nu}$ as a function $V^*_1$ by the rule we mentioned above. 
For $(x, y)\in V^*_0$, 
\begin{eqnarray}
|a_{0, \nu}(x, y)|
&=&|p_\nu(x)+q_\nu(y)+r_\nu| \leq |p_\nu(x)+r_\nu|+|q_\nu(y)+r_\nu|+|r_\nu| \nonumber \\
&\leq&\sup_{x\in U^+\cap U^*_0}|a_{0, \nu}(x)|+\sup_{y\in U^-\cap U^*_0}|a_{0, \nu}(y)|+|a_{0, \nu}(0, 0)| \leq 3A_\nu.\nonumber 
\end{eqnarray}
Thus we can regard $F_j$ as a holomorphic function defined on $V_j^*$. 
By construction, we have that 
$F_0|_{V^*_\pm}-F_1|_{V^*_\pm}=F_\pm|_{V^*_\pm}$ holds on $V^*_\pm$. 
\qed

\begin{remark}\label{rmk:R}
In Step 2 of the proof above, we applied Cauchy's inequality in several times, 
in which we used the fact that the circle $\{(z, w_1)\in V^*_1\mid z=z_0,\ |w_1|=1/R\}$ is included in $V^*_0$ for each $z_0\in U^*_\pm$. 
For this, we need to choose $V^*_j$'s and its coordinates appropriately as we did in 
\S \ref{section:V_tilde} and at the beginning of the proof. 
One of the most important property of our coordinates is that the projection $(z, w_1)\mapsto z$ coincides with $(x, y)\mapsto x$ on $V^*_+$ and with $(x, y)\mapsto y$ on $V^*_-$. 
On the other hand, we used an open covering of a neighborhood of $C$ taken by using a general theory (Siu's theorem \cite{Siu}) in \cite{K6}. 
Here we had to refine and shrink the open sets in order to take $R$ as a constant, see also \cite[Remark 4.3]{K6}. 
We here remark that one can slightly simplify the proof of 
\cite[Theorem 1.4]{K6} by replacing the open covering with $\{V^*_j\}$ we used in the present paper. 
\end{remark}

\begin{lemma}[{\cite[\S 4.2.3, 4.2.4]{K6}}]\label{lem:ueda_type_estim}
Let $n$ be a positive integer, $b_\pm$ a holomorphic function on $U_\pm^*$, and 
$a_j$ be a function on $U_j^*$ for $j=0, 1$ such that 
\[
\begin{cases}
     t_+^{-n}\cdot a_0-a_1= b_+ & (\text{on}\ U_+^*) \\
     t_-^{-n}\cdot a_0-a_1= b_- & (\text{on}\ U_-^*). 
  \end{cases}
\]
Then there exists a constant $K=K(C, \{U^*_j\})$ which does not depend on neither $n$, $a_j$ nor $b_\pm$ such that 
\[
\max_{j=0, 1}\sup_{U^*_j}|a_j|
\leq \frac{K}{|1-t^n|}\cdot \max\left\{\sup_{x\in U_+^*}|b_+(x)|,\ \sup_{y\in U_-^*}|b_-(y)|\right\}
\]
holds. 
\end{lemma}

In the rest of this subsection, we give a proof of Lemma \ref{lem:ueda_type_estim} for the convenience of the reader, although its statement is nothing but a summary of some arguments in \cite[\S 4.2.3, 4.2.4]{K6} intrinsically. 
Note that $t^n\not=1$ and $H^0(C, N_{C/S}^{-n})=H^1(C, N_{C/S}^{-n})=0$ hold (as we mentioned in \S \ref{section:Pic}), since 
$N_{C/S}$ is non-torsion. 
Therefore, $a_j$'s are uniquely determined by $b_\pm$. 

\begin{proof}[Proof of Lemma \ref{lem:ueda_type_estim}]
Set $M:=\max\left\{\sup_{x\in U_+^*}|b_+(x)|,\ \sup_{y\in U_-^*}|b_-(y)|\right\}$. 
Let $r$ be the value of $a_0$ at the nodal point. 
Then there uniquely exists a function $p$ on $U^+\cap U^*_0$ and $q$ on $U^-\cap U^*_0$ such that 
\[
a_0=\begin{cases}
     p+r & (\text{on}\ U^+\cap U^*_0) \\
     q+r & (\text{on}\ U^-\cap U^*_0). 
  \end{cases}
\]
Define $1$-forms $\omega_0$ and $\omega_1$ by 
\[
\omega_0:=da_0=\begin{cases}
     p'(x)dx & (\text{on}\ U^+\cap U^*_0) \\
     q'(y)dy & (\text{on}\ U^-\cap U^*_0) 
  \end{cases}
\]
and $\omega_1:=da_1=a_1'(z)dz$. 
By the assumption, we have that $t_\pm^{-n}\cdot \omega_0-\omega_1= db_\pm (=b_\pm'(z)dz)$ on $U_\pm^*$. 
Define a new open covering $\{U_j^{**}\}$ by 
\[
U^{**}_0:=\left\{(x, y)\in U^*_0\left| \max\{|x|, |y|\}<\frac{5\varepsilon}{3}\right.\right\},\ 
U^{**}_1:=\left\{(z, 0)\in U^*_1\left| \frac{4\varepsilon}{3} < |z| < \frac{3}{4\varepsilon}\right.\right\}. 
\]
As $U^{**}_j\Subset U^*_j$, we have that 
\[
\sup_{z\in U^\pm\cap U^{**}_{01}}|b_\pm'(z)|\leq K_1\cdot M
\]
holds on a constant $K_1>0$, where $U^{**}_{01}:=U^{**}_0\cap U^{**}_1$. 
By Lemma \ref{lem:KS_type_estim}, we have that 
\[
\max\left\{
\sup_{x\in {U^{**}_0}\cap {U^+}}|p'(x)|, 
\sup_{y\in {U^{**}_0}\cap {U^-}}|q'(y)|, 
\sup_{z\in {U^{**}_1}}|a_1'(z)|
\right\}
\leq K_0K_1M
\]
holds for a constant $K_0$. 
By considering the path integral from the nodal point, we have that 
\[
\max\left\{
\sup_{x\in {U^{**}_0}\cap {U^+}}|p(x)|, 
\sup_{y\in {U^{**}_0}\cap {U^-}}|q(y)|
\right\}
\leq \frac{5\varepsilon}{3} K_0K_1M. 
\]
By fixing point $z_\pm$ from 
$U^{**}_{01}\cap U^\pm$ and letting $C_\pm:=b_\pm(z_\pm)$ and $C_1:=a_1(z_+)$, respectively, we have that 
\[
b_\pm(z)=C_\pm+\int_{z_\pm}^zb_\pm'(\zeta)d\zeta,\ 
a_1(z)=C_1+\int_{z_+}^za_1'(\zeta)d\zeta. 
\]
Note that 
\[
\sup_{z\in U^{**}_1}\left|\int_{z_+}^za_1'(\zeta)d\zeta\right|
\leq K_2\cdot \sup_{z\in {U^{**}_1}}|a_1'(z)|
\leq K_0K_1K_2M
\]
holds for a constant $K_2$ which depends only on the diameter of $U^{**}_1$ (or equivalently, only on $\varepsilon$). 
As it follows 
\[
\begin{cases}
     t_+^{-n}\cdot (-p(z_+)+r)-C_1=C_+ \\
     t_-^{-n}\cdot (-q(z_-)+r)-\left(\int_{z_+}^{z_-} a_1'(z)dz + C_1\right)=C_-, 
  \end{cases}
\]
we have that 
$r=\frac{1}{t_+^{-n}-t_-^{-n}}\cdot\left(D_+-D_-\right)$ 
and 
$C_1=\frac{1}{t_+^{-n}-t_-^{-n}}\cdot\left(t_-^{-n}D_+-t_+^{-n}D_-\right)$, 
where 
$D_+:=t_+^{-n}p(z_+)+C_+$ and $D_-:=t_-^{-n}q(z_-)+C_-+\int_{z_+}^{z_-}a_1'(\zeta)d\zeta$. 
Note that 
\[
|D_+|\leq |p(z_+)|+|C_+|\leq\left(1+\frac{5\varepsilon}{3} K_0K_1\right)M
\]
and 
\[
|D_-|\leq |q(z_-)|+|C_-|+\sup_{z\in U^{**}_1}\left|\int_{z_+}^za_1'(\zeta)d\zeta\right|\leq \left(1+\frac{5\varepsilon}{3} K_0K_1+K_0K_1K_2\right)M. 
\]
Let us denote by $K_3$ the constant $2+\frac{10\varepsilon}{3} K_0K_1+K_0K_1K_2$. 
Then it follows from the arguments above that 
\[
\sup_{z\in U^{**}_1}\left|a_1(z)\right|
\leq |C_1|+
\sup_{z\in U^{**}_1}\left|\int_{z_+}^za_1'(\zeta)d\zeta\right|
\leq K_3\cdot \left(1+\frac{1}{|1-t^n|}\right)\cdot M
\]
and 
\[
\sup_{z\in U^{**}_0}\left|a_0(z)\right|
\leq K_3\cdot \left(1+\frac{1}{|1-t^n|}\right)\cdot M. 
\]
Thus we have 
\[
\max_{j=0, 1}\sup_{U^{**}_j}|a_j|< \frac{3K_3}{|1-t^n|}\cdot M. 
\]

When $z\in U^*_1\setminus U^{**}_1$, it holds that $z\in U^+\cap U^{**}_0$ or $z\in U^-\cap U^{**}_0$. 
In the former case, we have that 
\[
|a_1(z)|=|t_+^{-n}a_0(z)-b_+(z)|\leq |a_0(z)|+|b_+(z)|\leq 
\left(1+\frac{3K_3}{|1-t^n|}\right)\cdot M. 
\]
By the same arguments for the other cases, the lemma follows by letting $K:=2+3K_3$. 
\end{proof}

\begin{lemma}\label{lem:KS_type_estim}
Let $n$ be a positive integer 
and $i\colon\widetilde{C}\to C$ be the normalization such that the preimage of the nodal point is  $\{0, \infty\}\subset \mathbb{P}^1=\widetilde{C}$. 
Denote by $\widetilde{U^{**}_j}$ the preimage $i^{-1}(U^{**}_j)$ and $\widetilde{U^\pm}$ the preimage $i^{-1}(U^{\pm})$. 
Let $\eta_{\pm}$ be $1$-forms on $\widetilde{U^\pm}\cap \widetilde{U^{**}_{01}}$ such that 
the \v{C}ech cohomology class $[\{(\widetilde{U^\pm}\cap \widetilde{U^{**}_{01}},\eta_{\pm})\}]\in \check{H}^1(\{\widetilde{U^{**}_j}\}, K_{\widetilde{C}}\otimes i^*N_{C/S}^{-n})$ is trivial. 
Denote by $\omega_j$ the $1$-form on $\widetilde{U_j^{**}}$ for $j=0, 1$ uniquely determined by 
\[
\begin{cases}
     t_+^{-n}\cdot \omega_0-\omega_1= \eta_+ & (\text{on}\ \widetilde{U^+}\cap \widetilde{U^{**}_{01}}) \\
     t_-^{-n}\cdot \omega_0-\omega_1= \eta_- & (\text{on}\ \widetilde{U^-}\cap \widetilde{U^{**}_{01}}). 
  \end{cases}
\]
Then there exists a constant $K_0=K_0(C, \{U^{**}_j\})$ which does not depend on neither $n$ nor $\eta_\pm$ such that 
\begin{eqnarray}
&&\max\left\{
\sup_{x\in \widetilde{U^{**}_0}\cap \widetilde{U^+}}|g^+_0(x)|, 
\sup_{y\in \widetilde{U^{**}_0}\cap \widetilde{U^-}}|g^-_0(y)|, 
\sup_{z\in \widetilde{U^{**}_1}}|g_1(z)|
\right\} \nonumber \\
&\leq& K_0\cdot \max\left\{\sup_{z\in \widetilde{U^+}\cap \widetilde{U^{**}_{01}}}|h_+(z)|,\ \sup_{z\in \widetilde{U^-}\cap \widetilde{U^{**}_{01}}}|h_-(z)|\right\}, \nonumber
\end{eqnarray}
where $\omega_1=g_1(z)dz$, 
\[
\omega_0=\begin{cases}
     g^+_0(x)dx & (\text{on}\ \widetilde{U^+}\cap \widetilde{U^{**}_{01}}) \\
     g^-_0(y)dy & (\text{on}\ \widetilde{U^-}\cap \widetilde{U^{**}_{01}}), 
  \end{cases}
\]
and $\eta_\pm=h_\pm(z)dz$. 
\end{lemma}

\begin{proof}
By replacing $\omega_0$ with 
\[
\begin{cases}
     t_+^{-n}\cdot \omega_0 & (\text{on}\ \widetilde{U^+}\cap \widetilde{U^*_{01}}) \\
     t_-^{-n}\cdot \omega_0 & (\text{on}\ \widetilde{U^-}\cap \widetilde{U^*_{01}}), 
  \end{cases}
\]
the proof of the lemma is reduced to the case of $n=0$, which follows from \cite[Lemma 2]{KS}. 
\end{proof}

\begin{remark}\label{rmk:ueda_type_estim_general}
Lemma \ref{lem:ueda_type_estim} also holds in the case where $C$ is a cycle of multiple rational curves (see \cite[\S 4.2.3, 4.2.4]{K6} for details). 
Note that \cite[Lemma 4.2]{K6} is used for the estimate of the constants appears in the proof of the general statement which corresponds to the constant $C_1$ and $r$ in the proof above. 
\end{remark}

\subsection{Proof of Proposition \ref{prop:H^1} when $C$ is a cycle of multiple rational curves}

Let $C$ be a cycle of rational curves consists of $n$ irreducible components ($n\geq 2$). 
As Proposition \ref{prop:H^1} for this $C$ is shown by intrinsically the same arguments as in the previous section, here we only explain the outline. 

Denote by $C_{(1)}, C_{(2)} \dots, C_{(n-1)}, C_{(n)}=C_{(0)}$ the irreducible components of $C$. 
For $\nu=0, 1, 2, \dots, n-1$, Fix a small neighborhood $V_\nu$ of $C_{(\nu)}\cap C_{\rm reg}$ and 
$V_{\nu, \nu+1}$ of $C_{(\nu)}\cap C_{(\nu+1)}$. 
We may assume that $V_\nu\subset i(\widetilde{V}_{(\nu)})$, and 
that $V_{\nu, \nu+1}$ is included in the image of 
$\widetilde{V}_{0 (\nu+1)}^+=F_{\nu+1, \nu}^{-1}(\widetilde{V}_{0 (\nu)}^-)$ by $i$, where we are using the notation in Remark \ref{rmk:V_tilde_general}. 
Define coordinates $(z_\nu, w_\nu)$ of $V_\nu$ by 
$i^*z_\nu=S_{(\nu)}$ and $i^*w_\nu=S_{(\nu)}\cdot \xi_{0 (\nu)}=T_{(\nu)}\cdot \xi_{\infty (\nu)}$, 
and $(x_{\nu}, x_{\nu+1})$ of $V_{\nu, \nu+1}$ by $i^*x_{\nu+1}=S_{(\nu+1)}$ and 
$i^*x_\nu=F_{\nu+1, \nu}^*T_{(\nu)}$. 
Let 
\[
F^+(z_\nu, w_\nu)=\sum_{n=1}^\infty b^+_{\nu, \nu+1, n}(z_\nu)\cdot w_{\nu}^n
\]
be a holomorphic function defined on $V_\nu\cap V_{\nu, \nu+1}$, and 
\[
F^-(z_{\nu+1}, w_{\nu+1})=\sum_{n=1}^\infty b^-_{\nu, \nu+1, n}(z_{\nu+1})\cdot w_{\nu+1}^n
\]
be a holomorphic function defined on $V_{\nu+1}\cap V_{\nu, \nu+1}$. 
Then it is sufficient to find 
a holomorphic function $F_\nu$ on $V_\nu$ and $F_{\nu, \nu+1}$ on $V_{\nu, \nu+1}$ such that 
\[
\begin{cases}
 F_{\nu, \nu+1}-F_\nu = F^+ & (\text{on}\ V_\nu\cap V_{\nu, \nu+1})\\
 F_{\nu, \nu+1}-F_{\nu-1} = F^-  & (\text{on}\ V_{\nu+1}\cap V_{\nu, \nu+1})
\end{cases}
\]
by shrinking $V$. $F_\nu$ is constructed in the form of 
\[
F_\nu(z_\nu, w_\nu)=\sum_{n=1}^\infty a_{\nu, n}(z_\nu)\cdot w_{\nu}^n, 
\]
and $F_{\nu, \nu+1}$ is of 
\[
F_{\nu, \nu+1}(x_\nu, x_{\nu+1})=\sum_{n=1}^\infty a_{\nu, \nu+1, n}(x_\nu, x_{\nu+1})\cdot w_{\nu, {\nu+1}}(x_\nu, x_{x_\nu, x_{\nu+1}})^n, 
\]
where $w_{\nu, {\nu+1}}$ is the function defined by $i^*w_{x_\nu, x_{\nu+1}}=S_{(\nu)}\cdot \xi_{0 (\nu)}$, and the functions $a_{\nu, n}(z_\nu)$ and $a_{\nu, \nu+1, n}$ are holomorphic functions defined on $C\cap V_\nu$ and $C\cap V_{\nu, \nu+1}$, respectively. 
Let $p^{\nu+1}_{\nu, n}(x_\nu)$ be a function on $C_{(\nu)}\cap V_{\nu, \nu+1}$ and 
$p^{\nu}_{\nu+1, n}(x_{\nu+1})$ be a function on $C_{(\nu+1)}\cap V_{\nu, \nu+1}$ such that 
\[
a_{\nu, \nu+1, n}(x_\nu, x_{\nu+1})=\begin{cases}
p^{\nu+1}_{\nu, n}(x_\nu)+r_{\nu, \nu+1, n} & (\text{on}\ C_{(\nu)}\cap V_{\nu, \nu+1}) \\
p^{\nu}_{\nu+1, n}(x_{\nu+1})+r_{\nu, \nu+1, n} & (\text{on}\ C_{(\nu+1)}\cap V_{\nu, \nu+1})\\
\end{cases}
\]
holds, where $r_{\nu, \nu+1, n}:=a_{\nu, \nu+1, n}(0, 0)$. 
The function $a_{\nu, \nu+1, n}$ is also regarded as a function defined on $V_{\nu, \nu+1}$ by 
$a_{\nu, \nu+1, n}(x_\nu, x_{\nu+1}):=p^{\nu+1}_{\nu, n}(x_\nu)+p^{\nu}_{\nu+1, n}(x_{\nu+1})+r_{\nu, \nu+1, n}$. 
By setting $t^+_{\nu, \nu+1}:=1$ and $t^+_{\nu, \nu+1}:=t_{\nu+1, \nu}$, we have that 
\[
\begin{cases}
w_\nu=t^+_{\nu, \nu+1}\cdot w_{\nu, \nu+1} & (\text{on}\ C_{(\nu)}\cap V_{\nu, \nu+1}) \\
w_{\nu+1}=t^-_{\nu, \nu+1}\cdot w_{\nu, \nu+1} & (\text{on}\ C_{(\nu+1)}\cap V_{\nu, \nu+1})\\
\end{cases}
\]
holds. 

By the same argument as in Step 1 of the proof of Proposition \ref{prop:H^1}, it follows that one should define $a_{\nu, n}$'s and $a_{\nu, \nu+1, n}$'s by 
\[
\begin{cases}
    (t_{\nu, \nu+1}^+)^{-n}a_{\nu, \nu+1, n+1}-a_{\nu, n}=b^+_{\nu, \nu+1, n}-h^+_{\nu, \nu+1, n} & (\text{on}\ C_{(\nu)}\cap V_{\nu, \nu+1}) \\
    (t_{\nu, \nu+1}^-)^{-n}a_{\nu, \nu+1, n+1}-a_{\nu, n+1}=b^-_{\nu, \nu+1, n}-h^-_{\nu, \nu+1, n}(z)   & (\text{on}\ C_{(\nu+1)}\cap V_{\nu, \nu+1}). 
  \end{cases}
\]
Here the functions $h^\pm_{\nu, \nu+1, n}(z_\nu)$ are defined by 
\[
h^+_{\nu, \nu+1, n}(z_\nu)
=\sum_{m=1}^{n-1}(t_{\nu, \nu+1}^+)^{-m}\cdot P^{\nu}_{\nu+1, n, n-m}(z_\nu)
\]
and 
\[
h^-_{\nu, \nu+1, n}(z_{\nu+1})
=\sum_{m=1}^{n-1}(t_{\nu, \nu+1}^-)^{-m}\cdot P^{\nu+1}_{\nu, n, n-m}(z_{\nu+1}), 
\]
where 
\[
p^{\nu}_{\nu+1, n}(x_{\nu+1}(z_\nu, w_\nu))
=\sum_{\lambda=1}^\infty P^{\nu}_{\nu+1, n, \lambda}(z_\nu)\cdot w_\nu^\lambda
\]
and 
\[
p^{\nu+1}_{\nu, n}(x_{\nu}(z_{\nu+1}, w_{\nu+1}))
=\sum_{\lambda=1}^\infty P^{\nu+1}_{\nu, n, \lambda}(z_{\nu+1})\cdot w_{\nu+1}^\lambda. 
\]
As one can estimate $|a_{\nu, n}|$ and $|a_{\nu, \nu+1, n}|$ by the same argument as in Step 2 of the proof of Proposition \ref{prop:H^1}, the proposition holds (see also Remark \ref{rmk:ueda_type_estim_general}). 
\qed

\section{Proof of Theorem \ref{thm:arnold_for_nodal_rational}}\label{section:prf}

\subsection{Proof of Theorem \ref{thm:arnold_for_nodal_rational} when $C$ is a rational curve with a node}

Let $C$ be a rational curve with a node embedded in $S$ such that the normal bundle satisfies the Diophantine assumption in Theorem \ref{thm:arnold_for_nodal_rational}. 
We the notation in \S \ref{section:V_tilde}. 
Then it is sufficient to show that we may assume $G\equiv 1$ by changing the coordinates such as $S$ and $T$. 
Let $g(S, \xi_0):=\frac{1}{2\pi\sqrt{-1}}\log G(S, \xi_0)$ be the branch such that $g(0, 0)=0$. 
By applying Proposition \ref{prop:H^1_main} to 
$F_+:=-(V_0\to\widetilde{V}^+_0)^* g$ and 
$F_-:=0$, we have that, by shrinking $\widetilde{V}$ if necessary, 
there exist holomorphic functions $h_+\colon \widetilde{V}^+_0\to \mathbb{C}$, 
$h_1\colon \widetilde{V}_1\to \mathbb{C}$ and 
$h_-\colon \widetilde{V}^-_0\to \mathbb{C}$ such that 
\[
\begin{cases}
    h_+-h_1=-g & (\text{on}\ \widetilde{V}_1\cap \widetilde{V}^+_0) \\
    h_--h_1=0 & (\text{on}\ \widetilde{V}_1\cap \widetilde{V}^-_0) 
  \end{cases}
\]
holds (Set 
$h_+:=(\widetilde{V}^+_0\to V_0)^*F_0$, 
$h_-:=(\widetilde{V}^-_0\to V_0)^*F_0$ and 
$h_1:=(\widetilde{V}_1\to V_1)^*F_1$, for 
the solution $\{(V_j, F_j)\}$ in Proposition \ref{prop:H^1_main}). 
Define a function $h$ on $\widetilde{V}$ by
\[
h:=
\begin{cases}
    h_++g & (\text{on}\ \widetilde{V}_0^+) \\
    h_1 & (\text{on}\ \widetilde{V}_1) \\
    h_- & (\text{on}\ \widetilde{V}_0^-). 
  \end{cases}
\]
As clearly it holds that $F^*h_-=h_+$ by definition, we have that $F^*(h|_{\widetilde{V}_0^-})=h|_{\widetilde{V}_0^+}+g$. 
Denote by $H$ the function $e^{2\pi\sqrt{-1}h}$. 
Define a new coordinate function $\widehat{S}$ on $\widetilde{V}_0^+\cup \widetilde{V}_1$ by 
$\widehat{S}:=S\cdot H^{-1}$, 
$\widehat{T}$ on $\widetilde{V}_0^-\cup \widetilde{V}_1$ by 
$\widehat{T}:=T\cdot H$, 
$\widehat{\xi}_0$ on a neighborhood of $D_0$ by 
$\widehat{\xi}_0:=\widetilde{w}\cdot \widehat{S}^{-1}=\xi_0\cdot H$, 
and $\widehat{\xi}_\infty$ on a neighborhood of $D_\infty$ by 
$\widehat{\xi}_\infty:=\widetilde{w}\cdot \widehat{T}^{-1}=\xi_\infty\cdot H^{-1}$. 
Then, as it follows $F^*(H|_{\widetilde{V}_0^-})=H|_{\widetilde{V}_0^+}\cdot G$ by the construction, 
we have that 
\[
F^*\widehat{T}=(F^*T)\cdot (F^*H)=
\frac{t\cdot \xi_0}{G}\cdot (H\cdot G)
=t\cdot (\xi_0 H)
=t\cdot \widehat{\xi}_0
\]
and 
\[
F^*\widehat{\xi}_\infty
=(F^*\xi_\infty)\cdot (F^*H)^{-1}
=(G\cdot S)\cdot (H\cdot G)^{-1}
=S\cdot H^{-1}
=\widehat{S}
\]
on $F^{-1}(\widetilde{V}^-_0\cap \widetilde{V}_1)$. 
Therefore, by replacing $(S, \xi_0)$ and $(T, \xi_\infty)$ with 
$(\widehat{S}, \widehat{\xi}_0)$ and $(\widehat{T}, \widehat{\xi}_\infty)$ respectively, we have that 
$F(S, \xi_0)=(t\cdot \xi_0, S)$ holds, which proves the theorem. 
\qed

\subsection{Proof of Theorem \ref{thm:arnold_for_nodal_rational} when $C$ is a cycle of multiple rational curves}\label{section:prf_general}

Here we use the notation in Remark \ref{rmk:V_tilde_general}. 

First, we show that we may assumes that $G_{\nu+1, \nu}\equiv 1$ holds for $\nu=1, 2, n-2$ by changing the coordinates appropriately. 
Let $\{(\widetilde{V}_{(\nu)}', \widetilde{C}_{(\nu)}', (\widetilde{V}_{0 (\nu)}')^\pm, S_{(\nu)}', T_{(\nu)}', \xi_{0 (\nu)}', \xi_{\infty (\nu)}')\}_{\nu=1}^n$ be the $n$-copies of 
$(\widetilde{V}, \widetilde{C}, \widetilde{V}_{0}^\pm, S, T, \xi_{0}, \xi_{\infty})$ in Example \ref{ex:standard_model}. 
Denote by $i'\colon\coprod_{\nu=1}^n\widetilde{V}_{(n)}'\to \widetilde{V}'$ the quotient by the relation generated by the maps $(\widetilde{V}_{0 (\nu+1)}')^+\to (\widetilde{V}_{0 (\nu)}')^-$ naturally induced by $\widetilde{F}_{\nu+1, \nu}$'s for 
$\nu=1, 2, \dots, n-1$. 
In what follows, we regard $\widetilde{V}_{(\nu)}'$ as a subset of $\widetilde{V}'$. 
Note that $\widetilde{V}_{(1)}'\cap \widetilde{V}_{(n)}'=\emptyset$ holds as subset of $\widetilde{V}'$. 
Then it follows from a simple observation that the quotient $\widetilde{C}'$ of $\coprod_{\nu=1}^n\widetilde{C}_{(n)}'$ is a tree of rational curves with intersection matrix 
\[
\left(
    \begin{array}{ccccccc}
      -2 & 1 & 0 &\ldots & 0 & 0 & 0 \\
      1 &  -2& 1 &\ldots & 0  & 0 & 0 \\
      0 & 1 & -2 & \ddots& & \vdots& \vdots \\
      0 &0 &  \ddots&\ddots &\ddots & \vdots& \vdots \\
      \vdots & \vdots & & \ddots& -2& 1& 0 \\
      0 &0 & 0 &\ldots & 1& -2& 1 \\
      0 &0 & 0 & \ldots & 0 & 1 & -2
    \end{array}
  \right). 
\]
As this matrix is negative definite, it follows from \cite[Theorem 4.9]{L} and Grauert's theorem \cite{G} that 
$\widetilde{C}'$ admits a strictly pseudoconvex neighborhood $\widetilde{V}'$ whose maximal compact analytic subset is $\widetilde{C}'$. 
Note that, by the arguments as in \S \ref{prelim:general}, it holds that $H^1(\widetilde{C}', N_{\widetilde{C}'/\widetilde{V}'}^{-m})=0$ holds for each $m\geq 0$. 
Thus, it follows the same argument as in the proof of \cite[Proposition 3.1]{K2} that the restriction $H^1(\widetilde{V}', \mathcal{O}_{\widetilde{V}'})\to H^1(\widetilde{C}', \mathcal{O}_{\widetilde{C}'})$ is injective. 
As $H^1(\widetilde{C}', \mathcal{O}_{\widetilde{C}'})=0$, 
it follows from the same arguments as in the previous subsection that we may assume $G_{\nu+1, \nu}\equiv 1$ for $\nu=1, 2, n-2$. 

Therefore, the problem is reduced to showing that we may assume $G_{1, n}\equiv 1$ by changing the coordinates. 
By replacing $\xi_{0(0)}$ with $G_{1, n}(0, 0)^{-1/n}\cdot \xi_{0(0)}$, we may assume that $G_{1, n}(0, 0)=1$. 
Then the theorem follows by the same argument as in the previous subsection. 
\qed

\section{Toward the gluing construction of K3 surfaces corresponding to degenerations of K3 surfaces of type III}

Let 
$(C, S)$ be the example as in Example \ref{ex:ABU}, \ref{ex:ABU_2} or \ref{ex:ABU_3}. 
Assume that the normal bundles $N_{C/S}$ is a ${\rm U}(1)$-flat line bundles which satisfies Diophantine condition. 
Then it follows from Theorem \ref{thm:arnold_for_nodal_rational} that one can take a neighborhood $V$ of $C$ in $S$ 
as in Example \ref{ex:standard_model} or Example \ref{ex:standard_model_general} with $n=2$ or $3$. 
Define a function $\Phi\colon V\to\mathbb{R}$ by 
$i^*\Phi=|\widetilde{w}|$ when $C$ is a rational curve with a node, and 
by $(i^*\Phi)|_{\widetilde{V}_{(\nu)}}=|S_{(\nu)}\cdot \xi_{0 (\nu)}|=|T_{(\nu)}\cdot \xi_{\infty (\nu)}|$ 
when $C$ consists of two or three irreducible components. 

Fix positive constants $\delta$ and $R$ such that $R>1$ and $\delta<<1$. 
By the same argments as in \cite{KK3} we may assume 
$W:=\{p\in V\mid \Phi(p)<\delta R\}$ 
are relatively compact subsets of $V$ 
by shrinking $V$ 
and changing the scaling of the coordinates. 
Denote by 
$W^*$ the subset $\{p\in V\mid \delta/R<\Phi(p)<\delta R\}$ of $W$ 
and set 
$\widetilde{W}:=i^{-1}(W)$ and 
$\widetilde{W}^*:=i^{-1}(W^*)$, 
where $i\colon \widetilde{V}\to V$ 
is as in Example \ref{ex:standard_model} or \ref{ex:standard_model_general}. 
The set $W^*$ is a subset of 
$M:=S\setminus \{p\in V\mid \Phi(p)\leq \delta/R\}$. 

Define a meromorphic $2$-form $\eta_{\widetilde{W}}$ on 
$\widetilde{W}$ by 
\[
\eta_{\widetilde{W}}:=\frac{dS\wedge d\xi_0}{S\cdot \xi_0}=-\frac{dT\wedge d\xi_\infty}{T\cdot \xi_\infty}
\]
when $n=1$, and by 
\[
\eta_{\widetilde{W}}|_{\widetilde{W}\cap \widetilde{V}_{(\nu)}}:=\frac{dS_{(\nu)}\wedge d\xi_{0 (\nu)}}{S_{(\nu)}\cdot \xi_{0 (\nu)}}=-\frac{dT_{(\nu)}\wedge d\xi_{\infty (\nu)}}{T_{(\nu)}\cdot \xi_{\infty (\nu)}}
\]
for each $\nu$ when $n\geq 2$. 
As it holds that 
\[
F^*\eta_{\widetilde{W}}=-F^*\frac{dT}{T}\wedge F^*\frac{d\xi_\infty}{\xi_\infty}
=-\frac{d(t\xi_0)}{t\xi_0}\wedge \frac{dS}{S}
=\frac{dS\wedge d\xi_0}{S\cdot \xi_0}=\eta_{\widetilde{W}}
\]
when $n=1$ and 
\[
(F_{\nu+1, \nu})^*\eta_{\widetilde{W}}=-(F_{\nu+1, \nu})^*\frac{dT_{(\nu)}}{T_{(\nu)}}\wedge (F_{\nu+1, \nu})^*\frac{d\xi_{\infty (\nu)}}{\xi_{\infty (\nu)}}
=-\frac{d(t\xi_{0 (\nu)})}{t\xi_{0 (\nu)}}\wedge \frac{dS_{(\nu)}}{S_{(\nu)}}
=\frac{dS_{(\nu)}\wedge d\xi_{0 (\nu)}}{S_{(\nu)}\cdot \xi_{0 (\nu)}}=\eta_{\widetilde{W}}
\]
when $n\geq 2$, 
it follows that there exists a meromorphic $2$-form $\eta_W$ on $W$ with $i^*\eta_W=\eta_{\widetilde{W}}$ in both the cases. 
Now we have the following: 

\begin{proposition}\label{prop:mero1form}
$S$ admits a meromorphic $2$-form $\eta$ which has no zero and has poles only along $C$ such that 
$\eta|_{W}=\eta_W$ holds. 
\end{proposition}

Proposition \ref{prop:mero1form} is shown by the same argument as in the proof of \cite[Proposition 3.1]{KK3}. 
Here we use the fact that any leaf of a compact Levi-flat hypersurface of $W^*$ defined by $\{\widetilde{w}=\text{constant}\}$ is dense 
(Therefore, it follows that $H^0(W, \mathcal{O}_{W})\cong \mathbb{C}$ by the same arguments as in the proof of \cite[Lemma 3.2]{KK3}). 

\begin{proposition}\label{prop:K3_pi1}
It holds that $H_1(M, \mathbb{C})=0$. 
\end{proposition}

\begin{proof}
Take a real number $r$ with $\delta /R<r<\delta R$. 
As it is clear that $W^*$ is homotopic to $H_r:=\Phi^{-1}(r)$, 
it follows from Lemma \ref{lem:H} below that $H_1(W^*, \mathbb{C})\cong \mathbb{C}^2$. 
By Mayer--Vietoris sequence corresponds to the open covering $\{W, M\}$ of $S$, we obtain an exact sequence 
\[
H_2(S, \mathbb{C})\to H_1(W^*, \mathbb{C})\to H_1(W, \mathbb{C})\oplus H_1(M, \mathbb{C})\to H_1(S, \mathbb{C}). 
\]
As it is easily observed that the image of the map $H_2(S, \mathbb{C})\to H_1(W^*, \mathbb{C})$ is isomorphic to $\mathbb{C}$, we have that $H_1(M, \mathbb{C})=0$ (Note that $H_1(W, \mathbb{C})\cong \mathbb{C}$ and $H_1(S, \mathbb{C})=0$). 
\end{proof}

\begin{lemma}\label{lem:H}
The  Levi-flat manifold $H_r:=\Phi^{-1}(r)$ is $C^\omega$-diffeomorphic to 
$(\mathbb{C}^*\times{\rm U}(1))/\sim_{r, n}$ for sufficiently small $r$, 
where $\sim_{r, n}$ is the relation generated by 
\[
(\eta, \lambda)\sim_{r, n} (r^n\cdot \lambda^n\cdot \eta,\ t(N_{C/S})\cdot \lambda)
\]
for $(\eta, \lambda)\in \mathbb{C}^*\times{\rm U}(1)$. 
\end{lemma}

\begin{proof}
Let $\{(\widehat{V}_{(\nu)}, \widehat{C}_{(\nu)}, \widehat{V}_{0 (\nu)}^\pm, S_{(\nu)}, T_{(\nu)}, \xi_{0 (\nu)}, \xi_{\infty (\nu)})\}_{\nu=-\infty}^\infty$ be copies of 
$(\widetilde{V}, \widetilde{C}$, $\widetilde{V}_{0}^\pm$, $S$, $T$, $\xi_{0}$, $\xi_{\infty})$ in Example \ref{ex:standard_model}. 
Define a biholomorphism $F_{\nu+1, \nu}\colon \widehat{V}_{0 (\nu+1)}^+\to \widehat{V}_{0 (\nu)}^-$ by 
$(F_{\nu+1, \nu})^*(T_{(\nu)}, \xi_{\infty (\nu)})=(\xi_{0 (\nu)}, S_{(\nu)})$ 
for each $\nu$. 
Define $\widehat{V}$ by gluing $\widehat{V}_{(\nu)}$'s by $F_{\nu+1, \nu}$'s. 
Note that there is the natural covering map $\widehat{V}\to V$, which can be regarded as the universal covering. 

Consider the map $\widehat{g}_\nu\colon\{(\eta, \lambda)\in \mathbb{C}^*\times {\rm U}(1)\mid 2r^{-\nu+1}<|\eta|<2r^{-\nu-1}\}\to \widehat{V}_{(\nu)}$ defined by 
\[
(\widehat{g}_\nu)^*(S_{(\nu)}, \xi_{0 (\nu)})=\left(r^\nu\cdot \lambda^\nu\cdot \eta,\ \frac{1}{r^{\nu-1}\cdot \lambda^{\nu-1}\cdot \eta}\right). 
\]
Then, by a simple argument, it follows that $\widehat{g}_\nu$'s glue together to define an embedding $\widehat{g}\colon \mathbb{C}^*\times {\rm U}(1)\to \widehat{V}$. 
As it follows that $\widehat{g}$ induces an embedding $g\colon (\mathbb{C}^*\times{\rm U}(1))/\sim_{r, n}\to V$ and the image clearly coincides with $H_r$, the lemma follows. 
\end{proof}

Note that, by Lemma \ref{lem:H}, $H_r$ is diffeomorphic to $T^2_{g_n}$, which is the fiber bundle over ${\rm U}(1)$ whose fiber is $T^2:={\rm U}(1)\times {\rm U}(1)$ and the monodromy is 
$g_n\colon T^2\ni (p, q)\mapsto (pq^n, q)\in T^2$. 
From this, one can have that $H_1(H_r, \mathbb{Z})\cong\mathbb{Z}\oplus \mathbb{Z}\oplus (\mathbb{Z}/n\mathbb{Z})$.  
We emphasize that $H_r$ is not homeomorphic to $T^3:={\rm U}(1)\times {\rm U}(1)\times {\rm U}(1)$, since $H_1(T^3, \mathbb{Z})\cong \mathbb{Z}^{\oplus 3}$. 

By Proposition \ref{prop:mero1form} and \ref{prop:K3_pi1}, it seems to be natural to expect that one can construct a K3 surface by holomorphically gluing such models $\{(M_\nu, W^*_\nu, \eta_\nu|_{M_\nu})\}_\nu$ as obtained by the same construction as $(M, W^*, \eta|_M)$ 
(cf. \cite{KK3}). 

\begin{question}\label{q:K3}
Does there exist a non-singular K3 surface $X$ with holomorphic $2$-form $\sigma$ which admits an open covering $X=\bigcup_\nu M_\nu$ such that $\sigma|_{M_\nu}=\eta_\nu|_{M_\nu}$ and $M_\nu\cap M_\mu\subset W^*_\nu$ for each $\nu\not=\mu$, where $(M_\nu, W^*_\nu, \eta_\nu|_{M_\nu})$'s are as above? 
\qed
\end{question}

By considering the limit as the tab for gluing $\bigcup_\nu W^*_\nu$ goes to the set of zero measure (i.e. as $R\to 1$ and $\delta\to 0$), the K3 surfaces $X$ should degenerate to a singular K3 surface which is the union of rational surfaces and whose singular part is the union of a cycle of rational curves. 
We remark that the affirmative answer to Question \ref{q:K3} implies the existence of a K3 surface which includes a Levi-flat hypersurface which is diffeomorphic to $T^2_{g_n}$ (and thus is not homeomorphic to $T^3$). 



\end{document}